\documentclass[12pt]{amsart}
\usepackage{hyperref}
\usepackage{amssymb,xcolor}
\newcommand{\unfoldedcomment}[7]{}
\newcommand{\comment}[1]{}
\newcommand{\foldedcomment}[7]{}

\usepackage{mathtools}

\newcommand{\Ricmin}{\underline{\operatorname{Ric}}}
\DeclareMathOperator{\Sc}{Sc}
\DeclareMathOperator{\Div}{div}
\providecommand{\llbracket}{\mathopen{[\mkern-2.5mu[}}
\providecommand{\rrbracket}{\mathclose{]\mkern-2.5mu]}}

\newtheorem{theorem}{Theorem}[section]
\newtheorem{question}[theorem]{Question}
\newtheorem{corollary}[theorem]{Corollary}
\newtheorem{definition}[theorem]{Definition}
\newtheorem{lemma}[theorem]{Lemma}
\newtheorem{remark}[theorem]{Remark}
\newtheorem{proposition}[theorem]{Proposition}
\numberwithin{equation}{section}
\numberwithin{figure}{section}
\newtheorem{example}[theorem]{Example}

\begin{document}

\title[Warped area-minimizing hypersurfaces]{Warped Area-Minimizing Hypersurfaces and Warped Product Metrics}
\author{Yukai Sun}
\address{School of Mathematics and Statistics, Henan University, Kaifeng 475004 P. R. China and Center for Applied Mathematics of Henan Province, Henan University, Zhengzhou 450046 P. R. China
}
\email{sunyukai@henu.edu.cn}
\date{\today}

\begin{abstract}
Let $(M^n,g)$ be a smooth, compact, connected and oriented Riemannian manifold with $3\le n\le 7$ and $H_{n-1}(M;\mathbb{Z})\neq 0$, and let $u>0$ be smooth on $M$. Suppose that
\[
-\gamma u^{-1}\Delta_g u+\Ricmin_g\ \geq\ (n-1)(n-3)\,u^{-2}|\nabla u|^2,
\qquad \gamma=n-1,
\]
where $\Ricmin_g$ is the pointwise least Ricci curvature. Since both this hypothesis and the class of warped (weighted) area-minimizers are invariant under $u\mapsto\lambda u$, we normalize $u\equiv1$ on the minimizer. We prove that some non-trivial class in $H_{n-1}(M;\mathbb{Z})$ contains a warped area-minimizing hypersurface $\Sigma$ on which $u$ is constant, and that $M$ locally splits as a warped product $(-\epsilon,\epsilon)\times_f\Sigma$, $g=dt^2+f^2(t)g_\Sigma$, with $f=u^{-1}$ and $\operatorname{Ric}_{g_\Sigma}\geq (n-2)[(f')^2-ff'']\,g_\Sigma$. By extending the foliation via Federer--Fleming compactness of the constant-warped-area leaves, this local splitting upgrades to a \emph{global} one: $M$ is covered by $\mathbb{R}\times_f\Sigma$, its universal cover is $\mathbb{R}\times_f\tilde\Sigma$ where $\tilde\Sigma$ is the universal cover of $\Sigma$, and $f=u^{-1}$ is periodic and bounded above and below, so $M$ fibres over $\mathbb{S}^1$ with fibre a quotient of $\Sigma$. This spectral bound is saturated with equality by warped product metrics, which is what makes it the natural threshold for warped rather than direct product rigidity. When $u\equiv1$ it reduces to $\Ricmin_g\geq0$ and our conclusion to the classical product splitting.
\end{abstract}


\maketitle

\section{Introduction}

For compact manifolds, a classic rigidity result concerning area-minimizing hypersurfaces, due to Cai-Galloway \cite{cai-rigidity-2000}, asserts the following: if a three-dimensional manifold $(M, g)$ has non-negative scalar curvature $\Sc_g \geq 0$ and contains a two-sided, locally area-minimizing torus $\Sigma$, then $M$ is flat in a neighborhood of $\Sigma$; equivalently, $M$ admits a product metric near $\Sigma$. This confirms a conjecture originally proposed by Fischer-Colbrie-Schoen \cite{fischer-colbrie-structure-1980}. Separately, by the classification of area-minimizing surfaces in three-dimensional manifolds with nonnegative scalar curvature due to Schoen-Yau \cite{schoen-existence-1979}, any such surface in $M$ must be homeomorphic to either $\mathbb{S}^2$ (the 2-sphere) or $\mathbb{RP}^2$ (the real projective plane). Subsequently, Bray-Brendle-Neves \cite{bray-rigidity-2010} investigated the case of an area-minimizing 2-sphere $\Sigma$ in a compact three-dimensional manifold $M^3$ with scalar curvature bounded below by 2. In particular, they proved that the area of $\Sigma$ is at most $4\pi$; crucially, equality holds if and only if $M^3$ is covered by $\mathbb{S}^2 \times \mathbb{R}$. This implies that, in the rigid case, $M$ locally admits a product metric near $\Sigma$. Later, Zhu \cite{zhu2020rigidity} extended the Bray-Brendle-Neves result to a more general setting: for a closed manifold $(M^{m},g)$ with $3\leq m\leq 7$, $\Sc_g\geq 2$, admitting a non-zero degree map to $\mathbb{S}^2\times \mathbb{T}^{m-2}$, the area of an area-minimizing $2$-sphere is at most $4\pi$, with rigidity in the equality case.
Recently, Chu-Lee-Zhu \cite{chu-lee-zhu-2025-systole} generalized the Bray-Brendle-Neves result to the homological $k$-systole: for an orientable closed manifold $(M^{k+1},g)$ with $H_{k}(M)\neq0$ and bi-Ricci curvature $\operatorname{biRic}\geq k-1$, they proved $\operatorname{sys}_{k}(M,g)\leq|\mathbb{S}^{k}|$, and in the rigid case the universal cover of $M$ splits isometrically as $\mathbb{S}^{k}\times\mathbb{R}$. Their argument covers all dimensions, assuming the Generic Regularity Hypothesis, which is known to hold up to dimension ten and is not needed when $k+1\leq 7$.

For noncompact manifolds, there is also an extensive literature concerning metric splitting theorems in the presence of area-minimizing hypersurfaces. We refer the reader to:
\begin{itemize}
\item minimal surfaces in $3$-manifolds of nonnegative Ricci curvature \cite{anderson-minimal-1989};
\item stable minimal hypersurfaces in nonnegatively curved manifolds \cite{Li-Wang-Crelle-noncurved};
\item rigidity for complete manifolds with nonnegative scalar curvature \cite{zhu-JDG-2023};
\item effective versions of the positive mass theorem \cite{Carlotto-Chodosh-Eichmair-IM-effective};
\item foliations of area-minimizing hypersurfaces in asymptotically flat manifolds \cite{He-Shi-Yu-CVPDE};
\item a splitting theorem for scalar curvature \cite{chodosh-splitting-2019}.
\end{itemize}
We do not pursue the noncompact setting in this paper.

These classical rigidity statements mostly yield local \emph{direct-product} metric decompositions. A natural and largely unexplored question then arises:
\begin{question}\label{question-motivation}
Under suitable spectral Ricci curvature assumptions, can one derive local splitting into a genuine \emph{warped product metric}, rather than only a direct product?
\end{question}

Existing splitting theorems for spectral Ricci curvature typically produce isometric product decompositions. Moreover, as observed by Antonelli-Pozzetta-Xu \cite{antonelli-sharp-2024} and Catino-Mari-Mastrolia-Roncoroni \cite{catino-criticality-2025}, for a \emph{complete noncompact manifold with two ends} satisfying $-a\Delta_g u+\operatorname{Ric}_g\,u\geq 0$, product-type splitting may break down once the spectral parameter reaches the critical value $a\ge \frac{4}{n-1}$. We stress that those results concern a setting different from ours: the present paper deals with \emph{compact} manifolds, so the failure of product splitting for large $a$ on two-ended manifolds is not a statement that our hypotheses contradict, nor does our main theorem imply it. What the analogy does suggest, and what we verify directly in \eqref{eqn-spectral-saturated} below, is that warped product metrics saturate the spectral bound \eqref{ineqn-spectral-ricci} with equality, so that bound --- rather than plain nonnegativity of the spectral Ricci curvature --- is the condition naturally governing warped rigidity. This motivates our present work: we investigate whether local warped-product rigidity can be established from that saturated condition. To address this problem, we develop the variational theory of \emph{warped (weighted) area-minimizing hypersurfaces} and explore their geometric consequences.

Throughout this paper we write $\Delta_g=\operatorname{div}_g\nabla$ for the Laplacian, so that $-\Delta_g$ has nonnegative spectrum, and we set $\gamma=n-1$.

\begin{definition}\label{def spec}
  Let $(M, g)$ be a Riemannian manifold, let $u$ be a positive smooth function on $M$, and let $\gamma>0$. The \emph{spectral Ricci curvature} is the function
  \begin{equation} \label{eq spec ricci1}
    - \gamma u^{- 1} \Delta_g u + \Ricmin_g ,
  \end{equation}
where
  \begin{equation}
    \Ricmin_g := \inf_{e \in T_x M,\ |e|_g = 1} \operatorname{Ric}_g (e, e)
    \label{eq Rc}
  \end{equation}
  is the least Ricci curvature at the point $x \in M$.
\end{definition}

Spectral curvature plays an important role in the study of the stable Bernstein problem; see \cite{chodosh-stable-2024,mazet-stable-2024}. A variety of geometric results have been established concerning spectral curvature, including width estimates \cite{hirsch-spectral-2024,chai-band-2025,chai-Sun-2026}, diameter and volume estimates \cite{GC2021,AX2024,chu-spectral-2024,wu-spectral-2025}, and splitting theorems in the direct-product case \cite{antonelli-sharp-2024,catino-criticality-2025,Han-Wang-JFA-splitting,yeung-spectral-2025}.

To begin with, we analyze the Ricci curvature of the warped product manifold $(\mathbb{S}^{1}\times N^{n-1},g=dt^2+f^2(t)g_{N})$,
where $f$ denotes a positive smooth function on $\mathbb{S}^1$ and $(N^{n-1},g_{N})$ is a Riemannian manifold. The Ricci curvature of the metric $g$ is
\begin{align}\label{eqn-ricci-curvature}
    \operatorname{Ric}_{g}(\partial_{t},\partial_{t}) =& -(n-1)\frac{f''(t)}{f(t)}\\
    \operatorname{Ric}_{g}(X,Y) =& \operatorname{Ric}_{g_{N}}(X,Y)-\left[\frac{f''(t)}{f(t)}+(n-2)\frac{(f'(t))^2}{f^2(t)}\right]\langle X,Y\rangle_{g},\label{eqn-ricci-curvature2}
\end{align}
where $X,Y\in TN$.

Then, by direct computation, we have
\begin{align}\label{eqn-ricci-condition}
    -(n-1)f\Delta_{g}f^{-1}+\operatorname{Ric}_{g}(\partial_{t},\partial_{t})=&(n-1)(n-3)f^{-2}(f')^2.
\end{align}
The warped (or weighted) volume of a slice $N_t=\{t\}\times N$,
\begin{align}\label{eqn-constant-weighted-volume}
    E(N^{n-1})=\int_{N}f^{-(n-1)}d\operatorname{vol}_{g|_{N}}=\operatorname{Vol}_{g_{N}}(N),
\end{align}
is independent of $t$, since $d\operatorname{vol}_{g|_{N_t}}=f^{\,n-1}(t)\,d\operatorname{vol}_{g_N}$. Hence the first and second variations of $E(N^{n-1}_{t})$ along the $\partial_{t}$ direction both vanish.
Identity \eqref{eqn-ricci-condition} is precisely the spectral Ricci condition \eqref{ineqn-spectral-ricci} at the parameter $\gamma=n-1$, evaluated along the product direction, with equality. Therefore, we study warped minimal hypersurfaces under the spectral Ricci curvature condition.

We obtain splitting theorems for the \emph{warped-product case} under spectral-Ricci curvature conditions.

\begin{theorem}\label{Thm-ricci}
    Let $(M^{n},g)$ be a smooth, compact, connected and oriented Riemannian manifold with $3\leq n\leq 7$, let $u\in C^{\infty}(M)$ be positive, and assume
    \begin{align}\label{ineqn-spectral-ricci}
      -\gamma u^{-1}\Delta_{g}u+\Ricmin_{g}\ \geq\ (n-1)(n-3)\,u^{-2}|\nabla u|^2 ,
    \end{align}
    with $\gamma=n-1$, together with the homological non-vanishing hypothesis $H_{n-1}(M,\mathbb{Z})\neq 0$.

    Then there exists a non-trivial homology class $[\Sigma]\in H_{n-1}(M,\mathbb{Z})$, $[\Sigma]\neq 0$, containing a smooth embedded warped area-minimizing hypersurface $\Sigma$ in the sense of Definition \ref{def-warped-minimizer}. Moreover, after rescaling $u$ by a positive constant so that $u\equiv 1$ on $\Sigma$, the following hold:
    \begin{enumerate}
        \item $u$ is constant on $\Sigma$;
        \item $(M^{n},g)$ locally splits as $((-\epsilon,\epsilon)\times\Sigma,\ g=dt^2+f^{2}(t)g_{\Sigma})$ with $f(t)=u^{-1}(t)$, so that $f(0)=1$;
        \item for every $t\in(-\epsilon,\epsilon)$, the Ricci curvature of $\Sigma$ satisfies
        \[\operatorname{Ric}_{g_{\Sigma}}(\cdot,\cdot)\ \geq\ (n-2)\left[(f'(t))^2-f(t)f''(t)\right]g_{\Sigma}(\cdot,\cdot).\]
    \end{enumerate}
\end{theorem}

Conclusion (2) is \emph{local}: the warped product decomposition is asserted only on a small interval $(-\epsilon,\epsilon)$ about $\Sigma$. In Theorem \ref{Thm-global-splitting} we remove this restriction and show that the local splitting extends, with no further hypothesis, to a \emph{global} one on the universal cover: $\tilde M\cong\mathbb{R}\times_{f}\tilde\Sigma$, where $\tilde\Sigma$ is the universal cover of $\Sigma$, the warping function $f=u^{-1}$ is periodic and bounded above and below, and consequently $M$ fibres over $\mathbb{S}^{1}$ with fibre a quotient of $\tilde{\Sigma}$.

We emphasize that the lower bound in \eqref{ineqn-spectral-ricci} is not an ad hoc condition introduced to close the variational argument, but the \emph{natural} spectral threshold under which warped product splitting should be expected; it is \emph{saturated with equality} by warped product metrics. Indeed, let $(M^{n},g)$ be locally a warped product $(-\varepsilon,\varepsilon)\times_{f}\Sigma$, $g=dt^{2}+f^{2}(t)g_{\Sigma}$, and set $u:=f^{-1}$. Since $u$ depends only on $t$, identities \eqref{eqn-ricci-curvature}--\eqref{eqn-ricci-condition} yield, along the product direction $\partial_{t}$,
\begin{equation}\label{eqn-spectral-saturated}
  -\gamma\,u^{-1}\Delta_{g}u+\operatorname{Ric}_{g}(\partial_{t},\partial_{t})
  \;=\;(n-1)(n-3)\,u^{-2}|\nabla u|^{2},
\end{equation}
that is, equality is attained in \eqref{ineqn-spectral-ricci} along $\partial_{t}$. The bound thus plays, for warped splitting, the role that the equality case plays in the classical rigidity theorems of Cai--Galloway \cite{cai-rigidity-2000}, Bray--Brendle--Neves \cite{bray-rigidity-2010} and Zhu \cite{zhu2020rigidity}; for $n=3$ it reduces to the nonnegativity of the spectral Ricci curvature. In this sense, Theorem \ref{Thm-ricci} identifies the \emph{warped} analogue of nonnegative spectral Ricci curvature as the condition governing warped product rigidity.

\begin{remark}\label{rmk:naturalness}
The range $3\le n\le 7$ in Theorem \ref{Thm-ricci} comes from two independent sources, which should be kept apart.
\begin{itemize}
\item The upper bound $n\le 7$ is a \emph{regularity} restriction: an integral current of dimension $n-1$ that minimizes area in its homology class is supported on a smooth embedded hypersurface only up to dimension $6$, i.e.\ for $n\le 7$. This is the classical regularity theory of geometric measure theory.
\item The lower bound $n\ge 3$ is \emph{algebraic} and has nothing to do with regularity. In the second-variation estimate the coefficient of $|\nabla_{\Sigma}w|^{2}$ is $c:=\tfrac{\gamma}{4}-n+2=\tfrac{7-3n}{4}$ after the completion of squares in \eqref{eqn-completion}; $c<0$ is required for the resulting gradient term to have a definite sign, and $c<0$ holds exactly when $n\ge 3$. For $n=2$ one has $c=\tfrac14>0$ and the argument breaks down.
\end{itemize}
Moreover, identity \eqref{eqn-spectral-saturated} shows that the bound \eqref{ineqn-spectral-ricci} is saturated with equality along the product direction $\partial_t$ by warped products. We therefore describe \eqref{ineqn-spectral-ricci} as \emph{saturated}, or \emph{rigid}, rather than sharp: what we verify is that the right-hand side cannot be dropped, not that the constant $(n-1)(n-3)$ is optimal among all possible constants. For $n=3$ it reduces to nonnegativity of the spectral Ricci curvature (Remark \ref{rmk:n3}).
\end{remark}

\begin{remark}\label{rmk:stronger-than-spectral}
For $n>3$ the right-hand side $(n-1)(n-3)u^{-2}|\nabla u|^2$ of \eqref{ineqn-spectral-ricci} is nonnegative and positive wherever $\nabla u\neq0$. Hence \eqref{ineqn-spectral-ricci} is \emph{strictly stronger} than the usual assumption of nonnegative spectral Ricci curvature, $-\gamma u^{-1}\Delta_g u+\Ricmin_g\ge 0$, and the two coincide only when $n=3$ or when $u$ is constant. This strengthening is what makes the warped conclusion available: it is exactly the term that survives the completion of squares, cf.\ \eqref{eqn-completion} and identity \eqref{eqn-dim} in the proof of Proposition \ref{Prop-monotonic-formula}. Replacing $(n-1)(n-3)$ by any smaller constant would leave a residual term of indefinite sign in \eqref{eqn-dim} and would destroy the monotonicity formula \eqref{eq:H tilde inequality}.
\end{remark}

\begin{remark}\label{rmk:n3}
For $n=3$ condition \eqref{ineqn-spectral-ricci} reads $-2 u^{-1}\Delta_{g}u+\Ricmin_{g}\geq 0$, i.e.\ the spectral Ricci curvature is nonnegative. For a complete noncompact Riemannian manifold $(M^{n},g)$ with two ends satisfying
\[-a \Delta_{g}u+\operatorname{Ric}_{g}\,u\geq 0,\qquad a<\tfrac{4}{n-1},\]
Catino-Mari-Mastrolia-Roncoroni \cite{catino-criticality-2025} and Antonelli-Pozzetta-Xu \cite{antonelli-sharp-2024} proved that $M$ splits as a product manifold, and Antonelli-Pozzetta-Xu observed that such a product splitting fails to hold in general once $a\geq \frac{4}{n-1}$. Those results concern the complete noncompact two-ended setting and are therefore not directly comparable with Theorem \ref{Thm-ricci}, which assumes $M$ compact; we mention them only to indicate where the spectral threshold $\frac{4}{n-1}$ comes from. The point of comparison is the following: for $n=3$ and $\gamma=2$, the critical value $a=\frac{4}{n-1}=2$ coincides with our parameter, so in dimension three Theorem \ref{Thm-ricci} applies at precisely the parameter at which product splitting is known to fail in the two-ended setting, and yields a local warped product decomposition instead.

For nonnegative Bakry--\'{E}mery Ricci curvature, Liu \cite{liu-2013} obtains a locally split product metric by studying weighted area-minimizing hypersurfaces. Li-Wang \cite{Li-Wang-ASENS-weighted} also proved that the manifold splits as a warped product under a lower Ricci curvature bound and additional geometric conditions.
\end{remark}

The following example shows that the hypotheses of Theorem \ref{Thm-ricci} are satisfied by compact manifolds that are genuinely warped, hence not covered by direct-product splitting theorems.

\begin{example}\label{ex:warped-nonproduct}
Let $n\ge 3$, let $(\mathbb{S}^{n-1},g_{\mathbb{S}^{n-1}(R)})$ be the round sphere of radius $R>0$, and let $f\colon\mathbb{S}^{1}\to(0,\infty)$ be any positive smooth periodic function which is \emph{not} constant. Set
\[
M=\mathbb{S}^{1}\times \mathbb{S}^{n-1},\qquad g=dt^{2}+f^{2}(t)\,g_{\mathbb{S}^{n-1}(R)} .
\]
Then:
\begin{enumerate}
\item By the K\"unneth formula, $H_{n-1}(M;\mathbb{Z})\cong\mathbb{Z}$, generated by the slices $\{t\}\times\mathbb{S}^{n-1}$.
\item Set $u:=f^{-1}$. By \eqref{eqn-ricci-curvature}--\eqref{eqn-ricci-condition},
\begin{equation*}
-(n-1)f\Delta_g f^{-1}+\operatorname{Ric}_g(\partial_t,\partial_t)=(n-1)(n-3)f^{-2}(f')^{2},
\end{equation*}
and since $u^{-2}|\nabla u|^2=f^{2}\cdot f^{-4}(f')^2=f^{-2}(f')^2$, this reads
\begin{equation*}
-\gamma u^{-1}\Delta_g u+\operatorname{Ric}_g(\partial_t,\partial_t)=(n-1)(n-3)u^{-2}|\nabla u|^{2}.
\end{equation*}
For a unit vector $X\in T\mathbb{S}^{n-1}$, \eqref{eqn-ricci-curvature2} gives
\[
\operatorname{Ric}_g(X,X)=-\Big[\tfrac{f''}{f}+(n-2)\tfrac{(f')^{2}}{f^{2}}\Big]+\tfrac{n-2}{R^{2}f^{2}} ,
\]
so if $R>0$ is chosen small enough that $\frac{n-2}{R^{2}f^{2}}$ dominates, then $\Ricmin_g=\operatorname{Ric}_g(\partial_t,\partial_t)$ everywhere, and consequently \eqref{ineqn-spectral-ricci} holds on all of $M$ \emph{with equality}.
\item The warped volume of a slice is $E(\{t\}\times\mathbb{S}^{n-1})=\operatorname{Vol}(\mathbb{S}^{n-1}(R))$, independent of $t$; hence every slice is a warped area-minimizing hypersurface in its homology class, and $u=f^{-1}$ is constant on it.
\item Conclusion (3) reads $\operatorname{Ric}_{g_{\Sigma}}=\frac{n-2}{R^{2}}g_{\Sigma}\ge (n-2)[(f')^{2}-ff'']g_{\Sigma}$, which again holds for $R$ small.
\end{enumerate}
Thus $M$ satisfies all hypotheses of Theorem \ref{Thm-ricci}, and every conclusion is verified directly. Since $f$ was arbitrary and non-constant, $g$ is not a product metric: this exhibits a compact example of warped splitting in which direct-product splitting is unavailable.
\end{example}

The following degenerate case recovers the classical product splitting and confirms that Theorem \ref{Thm-ricci} is consistent with the rigidity theory recalled above.

\begin{corollary}\label{cor:u-constant}
Let $(M^{n},g)$ be a smooth, compact, connected and oriented Riemannian manifold with $3\le n\le 7$ and $H_{n-1}(M;\mathbb{Z})\neq0$, and assume $\Ricmin_g\ge 0$. Then there is a non-trivial class $[\Sigma]\in H_{n-1}(M;\mathbb{Z})$ represented by a smooth embedded area-minimizing hypersurface $\Sigma$, and $M$ locally splits isometrically as a \emph{direct} product $(-\epsilon,\epsilon)\times\Sigma$ with $g=dt^{2}+g_{\Sigma}$, where $\operatorname{Ric}_{g_{\Sigma}}\ge 0$.
\end{corollary}

\begin{proof}
Apply Theorem \ref{Thm-ricci} with $u\equiv1$. Then $\nabla u\equiv0$ and $\Delta_g u\equiv0$, so \eqref{ineqn-spectral-ricci} reduces to $\Ricmin_g\ge0$. Conclusion (1) is vacuous, conclusion (2) gives $f=u^{-1}\equiv1$, hence $g=dt^{2}+g_{\Sigma}$, and conclusion (3) gives $\operatorname{Ric}_{g_{\Sigma}}\ge(n-2)[(f')^{2}-ff'']g_{\Sigma}=0$.
\end{proof}

The proof of Theorem \ref{Thm-ricci} relies on the study of warped (or weighted) area-minimizing hypersurfaces. Specifically, we analyze the first and the second variations of such hypersurfaces, then employ a foliation argument analogous to that used in the theory of area-minimizing hypersurfaces \cite{bray-rigidity-2010,chai-band-2025} to establish the theorem.

The organization of this paper is as follows. In Section \ref{sec variation}, we set up the warped (or weighted) volume functional, explain why the exponent $\gamma=n-1$ makes it the area functional of a conformally related metric, and compute the first and second variations. In Section \ref{sec-Ric}, we present the proof of Theorem \ref{Thm-ricci} in six steps: the second-variation inequality together with a complete analysis of its equality case, the construction of a local foliation, a monotonicity formula, the proof that every leaf is again a stable warped area-minimizing hypersurface, the derivation of the warped product form of the metric, and finally the extension of this local splitting to a global warped product on the universal cover (Theorem \ref{Thm-global-splitting}). Section \ref{sec-conclusion} collects some concluding remarks and open questions, concerning the noncompact two-ended setting, the dimension restriction, and other values of the weight exponent $\gamma$.

\bigskip

{\em Acknowledgements}: The author would like to express his gratitude to Professor Yuguang Shi, Professor Thomas Richard and Professor Han Hong for their interest and conversation about this work. This work is partially supported by the National Natural Science Foundation of China (12601093).

\subsection*{Declaration of generative AI use}

An AI assistant was used for language editing and LaTeX preparation of this
manuscript. All mathematical results and proofs are due to the authors, who take
full responsibility for the content of the article.

\section{Basic facts on warped (weighted) minimal hypersurfaces}\label{sec variation}
Let $(M^n,g)$ be a smooth compact Riemannian manifold, let $\Sigma\subset M^n$ be a closed (compact, without boundary) embedded hypersurface, and let $u\in C^{\infty}(M)$ be a positive function. We define the warped (or weighted) volume of $\Sigma$ by the integral
\begin{equation}
  E (\Sigma) = \int_{\Sigma} u^{\gamma}
  \mathrm{d} \mathcal{H}^{n - 1}, \label{eq E}
\end{equation}
where $\gamma=n-1$ and $\mathcal{H}^{n - 1}$ denotes the $(n-1)$-dimensional Hausdorff measure induced by the metric $g$. Let $H$ and $A$ stand for the mean curvature and the second fundamental form of $\Sigma$, respectively.

The choice $\gamma=n-1$ is not arbitrary: it is exactly the exponent for which the functional $E$ becomes the ordinary area functional of a conformally related metric. Indeed, setting
\begin{equation}\label{eqn-conformal-metric}
  \tilde g := u^{2} g ,
\end{equation}
the induced $(n-1)$-dimensional area element scales as
\[
d\mathcal{H}^{n-1}_{\tilde g}=u^{\,n-1}d\mathcal{H}^{n-1}_{g}=u^{\gamma}d\mathcal{H}^{n-1},
\]
whence
\begin{equation}\label{eqn-E-is-area}
  E(\Sigma)=\mathcal{H}^{n-1}_{\tilde g}(\Sigma)=\operatorname{Area}_{\tilde g}(\Sigma) .
\end{equation}
Identity \eqref{eqn-E-is-area} has two immediate consequences, which we use throughout and which justify the regularity hypotheses in Theorem \ref{Thm-ricci}.

First, \emph{existence}: since $M$ is compact and $u>0$ is smooth, $u$ is bounded above and below by positive constants, so $E$ is comparable to $\mathcal{H}^{n-1}_g$. Given a non-trivial class $[\Sigma]\in H_{n-1}(M;\mathbb{Z})$, the Federer--Fleming compactness theorem for integral currents applied to the metric $\tilde g$ yields a $\tilde g$-area-minimizing integral current $T$ in that class; see \cite[Section 5.1.6]{FH1969}. By \eqref{eqn-E-is-area}, $T$ minimizes $E$ in its class.

Second, \emph{regularity and two-sidedness}: an integral current that minimizes mass is orientable, so its support is two-sided near its regular points, which gives the unit normal field $\nu$ used below. Moreover, since $\dim T=n-1\le 6$ whenever $n\le 7$, the classical regularity theory for area-minimizing currents implies that $\operatorname{spt}T$ is a smooth embedded hypersurface $\Sigma$. This is precisely the source of the restriction $n\le7$; it is a regularity statement and is independent of the algebraic restriction $n\ge3$ discussed in Remark \ref{rmk:naturalness}.

We record the definition formally.

\begin{definition}\label{def-warped-minimizer}
    Let $\Sigma$ be a closed hypersurface in a Riemannian manifold $(M,g)$, and suppose that its homology class $[\Sigma]\in H_{n-1}(M,\mathbb{Z})$ is non-zero. Then $\Sigma$ is called a \emph{warped area-minimizing hypersurface} if
    \[
    \int_{\Sigma} u^{\gamma}
  \mathrm{d} \mathcal{H}^{n - 1}=\inf_{\hat{\Sigma}\in [\Sigma]}\left\{\int_{\hat{\Sigma}} u^{\gamma}
  \mathrm{d} \mathcal{H}^{n - 1}\right\},
    \]
    the infimum being taken over all closed hypersurfaces $\hat\Sigma$ representing $[\Sigma]$.
\end{definition}

Note that both the condition \eqref{ineqn-spectral-ricci} and the class of warped area-minimizers are invariant under the rescaling $u\mapsto\lambda u$ for any constant $\lambda>0$: the left- and right-hand sides of \eqref{ineqn-spectral-ricci} are unchanged because both $\Delta_g u/u$ and $|\nabla u|/u$ are invariant, while $E$ is multiplied by the constant $\lambda^{\gamma}$, which does not affect which hypersurfaces minimize it in a fixed class. This is the normalization used in Theorem \ref{Thm-ricci}, where we rescale so that $u\equiv1$ on the minimizer $\Sigma$.

\begin{definition}\label{def-warped-minimal}
A closed hypersurface $\Sigma\subset M$ is called a \emph{warped (or weighted) minimal hypersurface} if
\[
H + \gamma u^{- 1} u_{\nu} = 0 \qquad\text{along } \Sigma .
\]
\end{definition}

\begin{lemma}[First variation]\label{lem-first-variation}
Let $\{\Sigma_t\}$, $t\in(-\delta,\delta)$, be a smooth one-parameter variation of $\Sigma=\Sigma_0$, with variation vector field $Y$, and let $\nu$ be a unit normal vector field along $\Sigma$. Then
\begin{equation}
  \left.\tfrac{\mathrm{d}}{\mathrm{d} t} E (\Sigma_t)\right|_{t = 0} = \int_{\Sigma} (H +
  \gamma u^{- 1} u_{\nu}) \langle Y, \nu \rangle u^{\gamma}\, \mathrm{d}
  \mathcal{H}^{n - 1} . \label{eq first var}
\end{equation}
In particular, every warped area-minimizing hypersurface in the sense of Definition \ref{def-warped-minimizer} is warped minimal in the sense of Definition \ref{def-warped-minimal}.
\end{lemma}

\begin{lemma}[Second variation]\label{lem-second-variation}
Let $\Sigma$ be a warped minimal hypersurface, let $\{\Sigma_{t}\}$ be a smooth one-parameter family of smooth hypersurfaces with $\Sigma_0=\Sigma$, variation vector field $Y$ and normal speed $\phi=\langle Y, \nu \rangle$. Then
\begin{align}
  \left.\tfrac{\mathrm{d}^2}{\mathrm{d} t^2} E (\Sigma_t)\right|_{t = 0}
=\ & \int_{\Sigma} \Bigl[- \Delta_{\Sigma} \phi - |A|^2 \phi
-\operatorname{Ric}_g(\nu, \nu) \phi \notag\\
&\qquad - \gamma u^{- 2} u_{\nu}^2 \phi
 + \gamma u^{- 1} \phi (\Delta_g u - \Delta_{\Sigma} u - H u_{\nu}) \notag\\
&\qquad - \gamma u^{- 1} \langle \nabla_{\Sigma} u,
\nabla_{\Sigma} \phi \rangle\Bigr] u^{\gamma} \phi\, d\mathcal H^{n-1}.
\label{eq second var}
\end{align}
\end{lemma}

\begin{definition}\label{def-stable}
A warped minimal hypersurface $\Sigma$ is called \emph{stable} if
\[\left.\tfrac{\mathrm{d}^2}{\mathrm{d} t^2} E (\Sigma_t)\right|_{t = 0}\geq 0\]
for every smooth one-parameter variation $\{\Sigma_t\}$ as in Lemma \ref{lem-second-variation}.
\end{definition}

\begin{lemma}\label{lm rewrite prelim}
The second variation \eqref{eq second var} can be rewritten as
\begin{align}
\left.\tfrac{\mathrm{d}^2}{\mathrm{d} t^2} E (\Sigma_t)\right|_{t = 0} & =
\int_{\Sigma} | \nabla_{\Sigma} \psi |^2 + \int_{\Sigma} \Bigl(\gamma \psi
\langle \nabla_{\Sigma} w, \nabla_{\Sigma} \psi \rangle +
(\tfrac{\gamma^2}{4} - \gamma) \psi^2 | \nabla_{\Sigma} w|^2\Bigr) \notag\\
& \quad + \int_{\Sigma} \bigl(\gamma u^{- 1} \Delta_g u - (|A|^2 + \operatorname{Ric}_g
(\nu, \nu))\bigr) \psi^2 \notag\\
& \quad - \int_{\Sigma} (\gamma Hw_{\nu} + \gamma w_{\nu}^2)
\psi^2 ,
\label{eq-cs25}
\end{align}
  where $\psi = \phi u^{\gamma / 2}$ and $w = \log u$.
\end{lemma}

\begin{proof}
See {\cite[Lemma 2.4]{chai-band-2025}}.

\end{proof}

\section{Proof of Theorem \ref{Thm-ricci}}\label{sec-Ric}

In this section we prove Theorem \ref{Thm-ricci}. The argument has six steps:
\begin{enumerate}
\item a second-variation inequality together with a complete analysis of its equality case (Lemma \ref{lem-rigidity-step1});
\item construction of a local foliation by hypersurfaces of constant $H_t+\gamma w_{\nu_t}$ (Lemma \ref{Lem-foliation-construct});
\item a monotonicity formula for $\tilde H(t)=H_t+\gamma w_{\nu_t}$ (Proposition \ref{Prop-monotonic-formula});
\item the conclusion that every leaf $\Sigma_t$ is again a stable warped area-minimizing hypersurface (Proposition \ref{Prop-minimiser});
\item writing $g$ locally as a warped product and deriving conclusions (2) and (3) (proof of Theorem \ref{Thm-ricci});
\item extending the local splitting to a global warped product on the universal cover, via compactness of the constant-$E$ leaves (Theorem \ref{Thm-global-splitting}).
\end{enumerate}

\subsection*{Step 1: the second variation and its equality case}

\begin{lemma}\label{lem-rigidity-step1}
Let $\Sigma$ be a stable warped area-minimizing hypersurface as in Theorem \ref{Thm-ricci}, and let $w=\log u$. Then
\begin{enumerate}
\item[(i)] $\nabla_{\Sigma}w\equiv 0$ on $\Sigma$; in particular $u$ is constant on each connected component of $\Sigma$;
\item[(ii)] $\Sigma$ is totally umbilical: $A=\frac{H}{n-1}\,g|_{\Sigma}$;
\item[(iii)] $\operatorname{Ric}_{g}(\nu,\nu)=\Ricmin_{g}$ on $\Sigma$;
\item[(iv)] equality holds in \eqref{ineqn-spectral-ricci} along $\Sigma$:
\[-\gamma u^{-1}\Delta_{g}u+\Ricmin_{g}=(n-1)(n-3)\,w_{\nu}^{2}\qquad\text{on } \Sigma .\]
\end{enumerate}
\end{lemma}

\begin{proof}
Let $\mathcal{S}$ denote the second variation of $E$ at $\Sigma$. Since $\Sigma$ is stable, $\mathcal{S}\geq 0$ for every admissible normal speed $\phi$, equivalently for every $\psi=\phi u^{\gamma/2}\in C^{\infty}(\Sigma)$. By Lemma \ref{lm rewrite prelim}, \eqref{eq-cs25} reads
\begin{align*}
\mathcal{S}=\ &\int_{\Sigma} | \nabla_{\Sigma} \psi |^2 + \int_{\Sigma} \Bigl(\gamma \psi
\langle \nabla_{\Sigma} w, \nabla_{\Sigma} \psi \rangle +
\bigl(\tfrac{\gamma^2}{4} - \gamma\bigr) \psi^2 | \nabla_{\Sigma} w|^2\Bigr)  \\
& + \int_{\Sigma} \bigl(\gamma u^{-1}\Delta_g u-|A|^2-\operatorname{Ric}_g(\nu,\nu)\bigr) \psi^2
 - \int_{\Sigma} \bigl(\gamma Hw_{\nu} + \gamma w_{\nu}^2\bigr)
\psi^2 .
\end{align*}

\medskip
\noindent\textbf{(a) The upper bound $\mathcal{S}\leq \mathcal{A}$.}  The spectral bound \eqref{ineqn-spectral-ricci} is equivalent to $\gamma u^{-1}\Delta_g u\leq \Ricmin_g-\gamma(n-3)u^{-2}|\nabla u|^2$. Since $\operatorname{Ric}_g(\nu,\nu)\geq\Ricmin_g$ and $|A|^2\geq \frac{H^2}{n-1}$, we obtain
\[
\gamma u^{-1}\Delta_g u-|A|^2-\operatorname{Ric}_g(\nu,\nu)\ \leq\ -\gamma(n-3)u^{-2}|\nabla u|^2-\tfrac{1}{n-1}H^2 .
\]
Let $\mathcal{A}$ denote the resulting upper bound; it is deliberately written in calligraphic font so as not to be confused with the second fundamental form $A$. Using $u^{-2}|\nabla u|^2=|\nabla w|^2=|\nabla_{\Sigma}w|^2+|w_{\nu}|^2$,
\begin{align*}
\mathcal{S}\ \leq\ \mathcal{A}:=\ &\int_{\Sigma} | \nabla_{\Sigma} \psi |^2
+ \int_{\Sigma} \gamma \psi \langle \nabla_{\Sigma} w, \nabla_{\Sigma} \psi \rangle \\
&+\int_{\Sigma}\left(\tfrac{\gamma}{4} -n+2\right)\gamma  \psi^2 | \nabla_{\Sigma} w|^2  \\
& - \int_{\Sigma} \left(\frac{1}{n-1}H^2 +\gamma Hw_{\nu} + \gamma(n-2) w_{\nu}^2\right)
\psi^2 .
\end{align*}
Here the coefficient of $|\nabla_\Sigma w|^2$ has been combined as $\tfrac{\gamma^2}{4}-\gamma-\gamma(n-3)=\gamma(\tfrac\gamma4-n+2)$, and the coefficient of $w_\nu^2$ as $-\gamma-\gamma(n-3)=-\gamma(n-2)$. Note that $\tfrac{\gamma}{4}-n+2=\tfrac{7-3n}{4}<0$ because $n\geq 3$; this sign is essential below.

\medskip
\noindent\textbf{(b) The gap $\mathcal{A}-\mathcal{S}$ is nonnegative.}  Subtracting,
\begin{equation}\label{eqn-gap}
\begin{aligned}
\mathcal{A}-\mathcal{S}=\ &\int_{\Sigma}\Bigl(|A|^2-\tfrac{1}{n-1}H^2\Bigr)\psi^2\\
&+\int_{\Sigma}\Bigl(\operatorname{Ric}_g(\nu,\nu)-\gamma u^{-1}\Delta_g u-\gamma(n-3)u^{-2}|\nabla u|^2\Bigr)\psi^2 .
\end{aligned}
\end{equation}
The first integrand is nonnegative by Cauchy--Schwarz. For the second, the spectral bound gives $\gamma u^{-1}\Delta_g u\leq \Ricmin_g-\gamma(n-3)u^{-2}|\nabla u|^2$, hence
\begin{equation*}
\operatorname{Ric}_g(\nu,\nu)-\gamma u^{-1}\Delta_g u-\gamma(n-3)u^{-2}|\nabla u|^2\ \geq\ \operatorname{Ric}_g(\nu,\nu)-\Ricmin_g\ \geq\ 0 .
\end{equation*}
Therefore $0\leq\mathcal{S}\leq \mathcal{A}$.

\medskip
\noindent\textbf{(c) Completion of squares.}  With $c:=\tfrac{\gamma}{4}-n+2=\tfrac{7-3n}{4}<0$,
\begin{align*}
&\gamma \psi
\langle \nabla_{\Sigma} w, \nabla_{\Sigma} \psi \rangle +
\gamma c\,\psi^2 | \nabla_{\Sigma} w|^2\\
=\ &\gamma c\left|\psi\nabla_{\Sigma}w+\frac{1}{2c}\nabla_{\Sigma}\psi\right|^2
-\frac{\gamma}{4c}|\nabla_{\Sigma}\psi|^2 ,
\end{align*}
so that, using $1-\frac{\gamma}{4c}=\frac{4n-8}{3n-7}$, $\gamma c=-\frac{(n-1)(3n-7)}{4}$ and $\frac{1}{2c}=\frac{2}{7-3n}$,
\begin{align}
\mathcal{A}=\mathcal{B}:=\ &\int_{\Sigma} \frac{4n-8}{3n-7}\,| \nabla_{\Sigma} \psi |^2 \notag\\
&-\int_{\Sigma}\frac{(n-1)\left(3n-7\right)}{4}\left|\psi\nabla_{\Sigma}w+\frac{2}{7-3n}\nabla_{\Sigma}\psi\right|^2 \notag\\
&- \int_{\Sigma} \left(\frac{1}{n-1}H^2 +\gamma Hw_{\nu} + \gamma(n-2) w_{\nu}^2\right)
\psi^2 .
\label{eqn-completion}
\end{align}
Since $\Sigma$ is warped minimal, $H=-\gamma w_{\nu}$, and the last integrand collapses:
\[
\tfrac{1}{n-1}H^2 +\gamma Hw_{\nu} + \gamma(n-2) w_{\nu}^2
=\gamma w_{\nu}^{2}\Bigl(\tfrac{\gamma}{n-1}-\gamma+n-2\Bigr)=0,
\]
because $\gamma=n-1$. Consequently $0\leq\mathcal{S}\leq \mathcal{A}=\mathcal{B}$, i.e.\ for every $\psi\in C^{\infty}(\Sigma)$,
\begin{equation}\label{eqn-key-ineq}
\begin{aligned}
&\int_{\Sigma} \frac{4n-8}{3n-7}\,| \nabla_{\Sigma} \psi |^2\\
&\qquad \geq\ \int_{\Sigma}\frac{(n-1)\left(3n-7\right)}{4}\left|\psi\nabla_{\Sigma}w+\frac{2}{7-3n}\nabla_{\Sigma}\psi\right|^2 .
\end{aligned}
\end{equation}

\medskip
\noindent\textbf{(d) The test function $\psi\equiv1$.}  This choice is admissible without cut-off: $\Sigma$ is closed, so $\phi:=u^{-\gamma/2}$ is a smooth function on $\Sigma$ and $\psi=\phi u^{\gamma/2}\equiv1$. Then $\nabla_{\Sigma}\psi\equiv0$ and \eqref{eqn-key-ineq} becomes
\[
0\ \geq\ \int_{\Sigma}\frac{(n-1)(3n-7)}{4}\,|\nabla_{\Sigma}w|^2\ \geq\ 0 ,
\]
since $3n-7>0$ for $n\geq3$. Hence $\nabla_{\Sigma}w\equiv0$, which proves (i).

For this test function $\mathcal{B}=0$, so $0\leq\mathcal{S}\leq \mathcal{A}=\mathcal{B}=0$ forces $\mathcal{S}=\mathcal{A}=0$ and therefore $\mathcal{A}-\mathcal{S}=0$. Both integrands in \eqref{eqn-gap} are nonnegative, so both vanish identically on $\Sigma$:
\begin{itemize}
\item $|A|^2=\frac{1}{n-1}H^2$, which by the equality case of Cauchy--Schwarz gives (ii);
\item $\operatorname{Ric}_g(\nu,\nu)-\gamma u^{-1}\Delta_g u-\gamma(n-3)u^{-2}|\nabla u|^2=0$. By the estimate in (b) this forces $\operatorname{Ric}_g(\nu,\nu)=\Ricmin_g$, which is (iii), and substituting back gives equality in \eqref{ineqn-spectral-ricci} along $\Sigma$.
\end{itemize}
Finally, (i) implies $|\nabla u|^2=u^2|\nabla w|^2=u^2|w_\nu|^2$ on $\Sigma$, so that equality takes the form in (iv).
\end{proof}

\begin{remark}\label{rmk:where-n-ge-3-enters}
Two features of \eqref{eqn-completion} depend on $n\geq3$ and fail for $n=2$. The coefficient $c=\frac{7-3n}{4}$ of the completed square must be negative, so that the square term enters \eqref{eqn-key-ineq} with a nonnegative coefficient; and $3n-7$ must be positive, so that $\frac{4n-8}{3n-7}>0$. For $n=2$ one has $c=\frac14>0$ and $\frac{4n-8}{3n-7}=0$, and the argument collapses. This is the algebraic origin of the restriction $n\geq3$ mentioned in Remark \ref{rmk:naturalness}, and it is independent of the regularity restriction $n\leq7$.
\end{remark}

\subsection*{Step 2: the local foliation}

The following argument is standard (see \cite{bray-rigidity-2010,zhu2021,chai-band-2025}).
\begin{lemma}\label{Lem-foliation-construct}
  Let \( \Sigma \) be a stable warped area-minimizing hypersurface satisfying the conclusions of Lemma \ref{lem-rigidity-step1}. Then, after possibly shrinking $\epsilon>0$, there is a local foliation $\{\Sigma_{t}\}_{-\epsilon< t< \epsilon}$ with $\Sigma_{0}=\Sigma$ such that
  \begin{enumerate}
      \item each $\Sigma_{t}$ is a graph over $\Sigma$ with graph function $\rho_{t}$ along the outward unit normal vector field $\nu$ such that
      \begin{equation}\label{eqn-rho-normalization}
          \left.\frac{\partial \rho_{t}}{\partial t}\right|_{t=0}=1 \quad \text{ and } \quad \frac{1}{\operatorname{vol}(\Sigma)}\int_{\Sigma}\rho_{t}\,dv=t;
      \end{equation}
      \item $H_{t}+\gamma w_{\nu_{t}}$ is constant on $\Sigma_{t}$.
  \end{enumerate}
  Here, \( d v \) is the volume element of \( \Sigma \) and $H_{t}$ is the mean curvature of $\Sigma_{t}$.
\end{lemma}
\begin{proof}
  Let
  \[\hat{C}^{\alpha}(\Sigma)=\left\{\phi\in C^{\alpha}(\Sigma):\ \int_{\Sigma}\phi\, dv =0\right\}\]
  for some $\alpha\in(0,1)$.
  Consider the map
  \begin{align*}
    \tilde{\Phi}:C^{2,\alpha}(\Sigma)&\longrightarrow \hat{C}^{\alpha}(\Sigma)\times \mathbb{R},\\
     \rho\longmapsto& \left(H_{\rho}+\gamma w_{\nu}-\frac{1}{\operatorname{vol}(\Sigma)}{\int_{\Sigma}\left(H_{\rho}+\gamma w_{\nu}\right)dv},\ \frac{1}{\operatorname{vol}(\Sigma)}{\int_{\Sigma}\rho\, dv}\right),
  \end{align*}
  where $H_{\rho}$ is the mean curvature of the graph over $\Sigma$ with graph function $\rho$.

  We calculate the first variation of $H + \gamma u^{- 1} u_{\nu} $ in the direction $\phi\nu$ at $t=0$. Every step below uses one of the conclusions (i)--(iv) of Lemma \ref{lem-rigidity-step1}:
\begin{align*}
  & \delta_{\phi \nu} (H + \gamma u^{- 1} u_{\nu}) \\
  = & - \Delta_{\Sigma} \phi - (\operatorname{Ric}_g(\nu,\nu) + |A|^2)
  \phi - \gamma u^{- 2} u_{\nu}^2 \phi - \gamma u^{- 1} \langle \nabla_{\Sigma} u, \nabla_{\Sigma} \phi \rangle \\
  & \quad + \gamma u^{- 1} (\Delta_g u - \Delta_{\Sigma} u - H u_{\nu}) \phi .
\end{align*}
By (i) we have $\nabla_\Sigma w\equiv0$, hence $\nabla_\Sigma u=u\nabla_\Sigma w\equiv0$. The last term in the preceding display therefore vanishes. Moreover, since the vector field $\nabla_\Sigma u$ vanishes identically \emph{along} $\Sigma$, its tangential covariant derivative along $\Sigma$ also vanishes, so that
\begin{equation}\label{eqn-DeltaSigma-u-vanishes}
\Delta_\Sigma u=\operatorname{div}_\Sigma(\nabla_\Sigma u)=0\qquad\text{on }\Sigma .
\end{equation}
Using also $\Sigma$ warped minimal, $u^{-2}u_\nu^2=w_\nu^2$ and $H=-\gamma w_\nu$, we obtain
\begin{align*}
  = & - \Delta_{\Sigma} \phi -\operatorname{Ric}_g (\nu,\nu) \phi - |A|^2 \phi - \gamma w_{\nu}^2 \phi \\
  & \quad + \gamma u^{- 1} \phi \Delta_g u - \gamma w_{\nu} ( - \gamma w_{\nu}) \phi .
\end{align*}
By (iii) and (iv), $-\operatorname{Ric}_g(\nu,\nu)\phi+\gamma u^{-1}\Delta_g u\,\phi=-\gamma(n-3)w_\nu^2\phi$. By (ii) together with $H=-\gamma w_\nu$, $|A|^2=\frac{H^2}{n-1}=\frac{\gamma^2 w_\nu^2}{n-1}$. Substituting these and \eqref{eqn-DeltaSigma-u-vanishes},
\begin{align*}
  = & - \Delta_{\Sigma} \phi  -\gamma(n-3)w_{\nu}^2\phi-\frac{\gamma^2w_\nu^2}{n-1}\phi+ \gamma^2 w_{\nu}^2 \phi -  \gamma w_{\nu}^2 \phi\\
  = & - \Delta_{\Sigma} \phi  -\gamma(n-3)w_{\nu}^2\phi+ \frac{n-2}{n-1}\gamma^2 w_{\nu}^2 \phi -  \gamma w_{\nu}^2\phi .
\end{align*}
Finally, with $\gamma=n-1$ the total coefficient of $w_\nu^2\phi$ is
\[
-\gamma(n-3)+\tfrac{n-2}{n-1}\gamma^2-\gamma=(n-1)\bigl[-(n-3)+(n-2)-1\bigr]=0,
\]
and therefore
\[
  \delta_{\phi \nu} (H + \gamma u^{- 1} u_{\nu}) = - \Delta_{\Sigma} \phi .
\]
Then the linearization of $\tilde{\Phi}$
at $\rho=0$,
  \begin{equation*}
    D\tilde{\Phi}|_{\rho=0}: C^{2,\alpha}(\Sigma)\longrightarrow \hat{C}^{\alpha}(\Sigma)\times \mathbb{R},\quad \psi\longmapsto \left(-\Delta_{\Sigma} \psi,\frac{1}{\operatorname{vol}(\Sigma)}\int_{\Sigma}\psi\, dv\right),
  \end{equation*}
  is invertible: $\ker(-\Delta_\Sigma)$ consists of the constants, and the second component maps a nonzero constant to a nonzero real number. By the inverse function theorem, we can find a family of functions $\rho_{t}:\Sigma\to \mathbb{R}$, $t\in (-\epsilon,\epsilon)$, with the following properties:
  \begin{enumerate}
    \item $\rho_{0}\equiv0$, and
    \begin{equation}\label{Eqn-construct-rho}
      \left.\frac{\partial \rho_{t}}{\partial t}\right|_{t=0}\equiv 1,\quad \text{and}\quad \frac{1}{\operatorname{vol}(\Sigma)}\int_{\Sigma}\rho_{t}\, dv =t;
    \end{equation}
  \item the graphs $\Sigma_{t}$ over $\Sigma$ with graph function $\rho_{t}$ have constant \( H_t + \gamma u^{-1} u_{\nu_t} \).
  \end{enumerate}
  From \eqref{Eqn-construct-rho}, after decreasing $\epsilon$ slightly, $\partial_t\rho_{t}>0$ for \( t \in (-\epsilon, \epsilon) \),
from which it follows that the graphs $\{\Sigma_{t}\}_{t\in (-\epsilon,\epsilon)}$ form a foliation of a neighborhood of $\Sigma$.
\end{proof}

\subsection*{Step 3: the monotonicity formula}

\begin{proposition}\label{Prop-monotonic-formula}
There exists a continuous function $\Psi (t)$ such
  that
  \begin{equation}
    \tfrac{\mathrm{d}}{\mathrm{d} t} \left( \exp \left(\int_0^t \Psi (\tau)
    \mathrm{d} \tau\right) \tilde{H} \right) \leq 0 \label{eq:H tilde inequality}
  \end{equation}
   where
  \begin{equation}
    \tilde{H}(t) = H_{t} + \gamma w_{\nu_{t}}. \label{H tilde}
  \end{equation}
\end{proposition}

\begin{proof}
    Let $\Phi : \Sigma \times (- \epsilon,
\epsilon) \to M$ parametrize the local foliation of Lemma \ref{Lem-foliation-construct}, let $Y = \tfrac{\partial
\Phi}{\partial t}$, and set $\phi_t = \langle Y, \nu_t \rangle$. By \eqref{eqn-rho-normalization} we have $\phi_0=\partial_t\rho_t|_{t=0}\equiv1$, so $\phi_0$ is a positive constant; by continuity, after shrinking $\epsilon$ we may assume
$\phi_t > 0$ for all $t \in (- \epsilon, \epsilon)$. Recall that the first variation gives
\begin{align*}
  & - \tilde{H}' (t) \\
  = & - \tfrac{\mathrm{d}}{\mathrm{d} t} (H_t + \gamma w_{\nu_{t}})
  \\
  = & \Delta_{\Sigma_t} \phi_t + (\operatorname{Ric}_g(\nu_t,\nu_t) +
  |A_t |^2) \phi_t + \gamma u^{- 2} u_{\nu_{t}}^2 \phi_t \\
  & \quad - \gamma u^{- 1} (\Delta_{g} u - \Delta_{\Sigma_t} u - H_{t} u_{\nu_{t}})
  \phi_t + \gamma u^{- 1} \langle \nabla_{\Sigma_t} u, \nabla_{\Sigma_t}
  \phi_t \rangle .
\end{align*}
We use three estimates. First, $\operatorname{Ric}_g(\nu_t,\nu_t)\geq\Ricmin_g$. Second, the spectral bound \eqref{ineqn-spectral-ricci} can be rewritten as
\[
\gamma u^{-1}\Delta_g u\ \leq\ \Ricmin_g-\gamma(n-3)u^{-2}|\nabla u|^2 .
\]
Third, Cauchy--Schwarz gives $|A_t|^2\geq\frac{H_t^2}{n-1}$. Combining these with the identity \eqref{H tilde}, we see
\begin{align*}
  & - \tilde{H}' (t) \phi_t^{- 1} \\
  \geq & (\phi_t^{- 1} \Delta_{\Sigma_t} \phi_t + \gamma  u^{- 1}
  \Delta_{\Sigma_t} u + \gamma u^{- 1} \langle \nabla_{\Sigma_t} u,
  \nabla_{\Sigma_t} \phi_t \rangle \phi_t^{- 1}) \\
  & \quad  +\frac{1}{n-1} H_{t}^2 + \gamma w_{\nu_{t}}^2 + \gamma H_{t} w_{\nu_{t}} +\gamma(n-3)u^{-2}|\nabla u|^2.
\end{align*}
Inserting \eqref{H tilde} in the above and using
\begin{equation}\label{eqn-dim}
\begin{aligned}
  & \phi_t^{- 1} \Delta_{\Sigma_t} \phi_t + \gamma  u^{- 1}
  \Delta_{\Sigma_t} u \\
  &+ \gamma u^{- 1} \langle \nabla_{\Sigma_t} u,
  \nabla_{\Sigma_t} \phi_t \rangle \phi_t^{- 1}+\gamma(n-3)u^{-2}|\nabla_{\Sigma_{t}} u|^2 \\
  = & \Div_{\Sigma_t} \left(
  \frac{\nabla_{\Sigma_t} \phi_t}{\phi_t} + \gamma \frac{\nabla_{\Sigma_t}
  u}{u} \right) + \frac{(n-1)(3n-7)}{4} \left| \tfrac{\nabla_{\Sigma_t}
  u}{u} \right|^2\\
  &\quad +  \left| \frac{\gamma}{2}\frac{\nabla_{\Sigma_t} u}{u} +
  \frac{\nabla_{\Sigma_t} \phi_t}{ \phi_t} \right|^2 \\
  \geq & \Div_{\Sigma_t} \left(
  \frac{\nabla_{\Sigma_t} \phi_t}{\phi_t} + \gamma \frac{\nabla_{\Sigma_t}
  u}{u} \right)
\end{aligned}
\end{equation}
yields
\begin{align*}
  & - \tilde{H}' (t) \phi_t^{- 1} \\
  \geq & \Div_{\Sigma_t} \left(
  \frac{\nabla_{\Sigma_t} \phi_t}{\phi_t} + \gamma \frac{\nabla_{\Sigma_t}
  u}{u} \right) +\gamma(n-3)u^{-2} u_{\nu_{t}}^2\\
  & + \frac{1}{n-1}\tilde{H}^2-\frac{2}{n-1}\tilde{H}
  \gamma w_{\nu_{t}}+\frac{1}{n-1}\gamma^2w_{\nu_{t}}^2 \\
  &+ \gamma w_{\nu_{t}}^2+\gamma\tilde{H}w_{\nu_{t}}-\gamma^2w_{\nu_{t}}^2\\
  = & \Div_{\Sigma_t} \left(
  \frac{\nabla_{\Sigma_t} \phi_t}{\phi_t} + \gamma \frac{\nabla_{\Sigma_t}
  u}{u} \right) \\
  & \quad + \frac{1}{n-1}\tilde{H}^2+\frac{n-3}{n-1}\tilde{H}
  \gamma w_{\nu_{t}}+\underbrace{\left(\frac{2-n}{n-1}\gamma+n-2\right)\gamma w_{\nu_{t}}^2}_{=\ 0\ \text{ since }\gamma=n-1}.
\end{align*}
Here we have written $u^{-2}|\nabla u|^2=|\nabla w|^2=|\nabla_{\Sigma_t}w|^2+w_{\nu_t}^2$ and kept only the normal part, which is the term \eqref{eqn-dim} already accounts for.

\smallskip
The vanishing of the underbraced coefficient deserves emphasis, because it is exactly what singles out the threshold $(n-1)(n-3)$ in \eqref{ineqn-spectral-ricci}. Collecting all the $w_{\nu_t}^2$ contributions --- the $\gamma(n-3)u^{-2}u_{\nu_t}^2$ term, the expansion of $\frac{1}{n-1}H_t^2$, and $\gamma w_{\nu_t}^2+\gamma H_tw_{\nu_t}$ --- gives
\[
\gamma\Bigl[(n-3)+\tfrac{\gamma}{n-1}+1-\gamma\Bigr]
=\gamma\Bigl[\tfrac{2-n}{n-1}\gamma+(n-2)\Bigr],
\]
which vanishes identically once $\gamma=n-1$. Had the constant on the right-hand side of \eqref{ineqn-spectral-ricci} been any other number $\kappa$ in place of $(n-1)(n-3)$, a residual term $(\kappa-(n-1)(n-3))\,u^{-2}u_{\nu_t}^2$ of \emph{indefinite} sign would survive here and the differential inequality below would fail. In other words, \eqref{ineqn-spectral-ricci} is the unique spectral lower bound for which the $w_\nu^2$ terms cancel.

With $\gamma=n-1$ the last inequality is therefore exact in the $w_{\nu_t}^2$ slot, and dropping only the manifestly nonnegative term $\frac{1}{n-1}\tilde H^2$ gives
\begin{align*}
  & - \tilde{H}' (t) \phi_t^{- 1} \\
  \geq & \Div_{\Sigma_t} \left( \frac{\nabla_{\Sigma_t}
  \phi_t}{\phi_t} + \gamma \frac{\nabla_{\Sigma_t} u}{u} \right)+\frac{n-3}{n-1}\tilde{H}
  \gamma w_{\nu_{t}}.
\end{align*}
We integrate the above on $\Sigma_t$; the divergence term drops out because $\Sigma_t$ is closed, and we find
\[ - \tilde{H}' (t) \int_{\Sigma_t} \tfrac{1}{\phi_t} \geq \tilde{H} (t)
   \int_{\Sigma_t} (n-3)
  w_{\nu_{t}}. \]
Setting $\Psi (t) = (\int_{\Sigma_t} \tfrac{1}{\phi_t})^{- 1} \int_{\Sigma_t}
(n-3)
  w_{\nu_{t}}$, which is continuous on $(-\epsilon,\epsilon)$ since $\phi_t>0$ and $\Sigma_t$ varies continuously, we obtain
\[ \tilde{H}' + \Psi (t) \tilde{H} \leq 0. \]
Multiplying by the positive integrating factor $\exp(\int_0^t\Psi(\tau)d\tau)$ gives \eqref{eq:H tilde inequality}, which completes the proof.
\end{proof}

\begin{remark}\label{rmk:n3-psi-vanishes}
In dimension $n=3$ the factor $(n-3)$ vanishes, so $\Psi\equiv0$ and Proposition \ref{Prop-monotonic-formula} degenerates to the plain monotonicity statement $\tilde H'\leq0$. No integrating factor is needed, and \eqref{eq:H tilde inequality} reduces to $\tilde H'(t)\leq0$ on $(-\epsilon,\epsilon)$.
\end{remark}

\subsection*{Step 4: every leaf is again a minimizer}

\begin{proposition}\label{Prop-minimiser}
For every $t\in(-\epsilon,\epsilon)$, the leaf $\Sigma_t$ of the foliation constructed in Lemma \ref{Lem-foliation-construct} is a stable warped area-minimizing hypersurface in $M$, and $\tilde H(t)\equiv0$ on $(-\epsilon,\epsilon)$.
\end{proposition}
\begin{proof}
  Let \( \partial_t \) be the variation vector field of the foliation \( \{\Sigma_t \}_{t\in (-\epsilon,\epsilon)} \), and \( \phi_t = \langle \partial_t , \nu_t \rangle>0 \).
    By \eqref{eq first var} the first variation of $E$ along the foliation is
\[ \tfrac{\mathrm{d}}{\mathrm{d} t} E (\Sigma_t) = \int_{\Sigma_t} \tilde{H}
   (t)\, u^{\gamma} \phi_{t}\, d\mathcal H^{n-1}. \]
Since $\Sigma_0$ is warped minimal, $\tilde{H}(0)=0$; and $\tilde H$ is constant on each $\Sigma_t$ by construction, so $\tilde H(t)$ is a genuine function of $t$ alone. It follows from \eqref{eq:H tilde inequality} that the function $t\mapsto \exp(\int_0^t\Psi)\tilde H(t)$ is nonincreasing and vanishes at $t=0$, hence
\[
\tilde{H} (t) \leq 0 \ \text{ for } t\geq 0,
\qquad
\tilde{H} (t) \geq 0 \ \text{ for } t\leq 0 .
\]
Therefore $\frac{d}{dt}E(\Sigma_t)\leq0$ for $t\geq0$ and $\frac{d}{dt}E(\Sigma_t)\geq0$ for $t\leq0$, i.e.\ $t\mapsto E(\Sigma_t)$ attains its maximum at $t=0$. So $E(\Sigma_t)\leq E(\Sigma_0)$ for all $t\in(-\epsilon,\epsilon)$.

On the other hand, each $\Sigma_t$ is homologous to $\Sigma_0$: the annulus $\Phi(\Sigma\times[0,t])$ is a singular $n$-chain whose boundary is $\Sigma_t-\Sigma_0$. Since $\Sigma_0$ is warped area-minimizing in its class, $E(\Sigma_0)\leq E(\Sigma_t)$. Combining the two inequalities,
\[ E (\Sigma_t) = E (\Sigma_0) \qquad\text{for all } t\in(-\epsilon,\epsilon). \]
Since $\frac{d}{dt}E(\Sigma_t)=\tilde H(t)\int_{\Sigma_t}u^\gamma\phi_t\,d\mathcal H^{n-1}$ and the integral is strictly positive (as $\phi_t>0$ and $u>0$), the constancy of $E(\Sigma_t)$ forces
\begin{equation}\label{eqn-Htilde-identically-zero}
\tilde{H}(t)\equiv 0 \qquad\text{on } (-\epsilon,\epsilon).
\end{equation}

We have thus shown that each $\Sigma_t$ is a closed hypersurface with $E(\Sigma_t)=E(\Sigma_0)=\inf$ in the class $[\Sigma_0]$, so $\Sigma_t$ is warped area-minimizing in the sense of Definition \ref{def-warped-minimizer}, and it is stable by Definition \ref{def-stable}. Moreover \eqref{eqn-Htilde-identically-zero} says that $\Sigma_t$ is warped minimal, so every hypothesis of Lemma \ref{lem-rigidity-step1} is satisfied by $\Sigma_t$ in place of $\Sigma_0$.
\end{proof}

\subsection*{Step 5: writing $g$ as a warped product}

We now complete the proof of Theorem \ref{Thm-ricci}.

\begin{proof}[Proof of Theorem \ref{Thm-ricci}, conclusion]
The existence of $\Sigma$, its smoothness and two-sidedness, and the validity of $3\leq n\leq7$ were established in Section \ref{sec variation}, cf.\ \eqref{eqn-E-is-area}. Conclusion (1) is Lemma \ref{lem-rigidity-step1}(i) together with the normalization $u|_\Sigma\equiv1$ explained after Definition \ref{def-warped-minimizer}. It remains to derive (2) and (3).

Since $\partial_t\rho_t>0$ on $(-\epsilon,\epsilon)$, the leaves $\{\Sigma_t\}$ fill a neighborhood $U$ of $\Sigma$ and are pairwise disjoint. Replacing $t$ by its arc-length reparametrization along the integral curves of $\nu_t$ --- which is possible because $\phi_t>0$, and which we continue to denote by $t$ --- we may assume $\partial_t=\nu_t$ on $U$. In these coordinates $U=(-\epsilon,\epsilon)\times\Sigma$ and
\[
g=dt^{2}+f^{2}(t)\,g_{\Sigma}
\]
for some positive function $f$ on $(-\epsilon,\epsilon)$ with $f(0)=1$, where $g_\Sigma$ denotes the induced metric on $\Sigma_0=\Sigma$. Indeed, $\partial_t$ is orthogonal to each leaf by construction, and each $\Sigma_t$ is totally umbilical by Proposition \ref{Prop-minimiser} applied with Lemma \ref{lem-rigidity-step1}(ii): the second fundamental form of $\Sigma_t$ is $A_t=\frac{H_t}{n-1}g_{\Sigma_t}$, so the metric on the leaves evolves by a pure conformal factor, $\partial_t g_{\Sigma_t}=2\frac{H_t}{n-1}g_{\Sigma_t}$. This integrates to $g_{\Sigma_t}=f^2(t)g_\Sigma$ with $f>0$ and $f(0)=1$.

By \eqref{eqn-Htilde-identically-zero}, $H_t=-\gamma w_{\nu_t}=-\gamma u^{-1}\partial_t u$. In the coordinates above $\partial_t u=u'(t)$ and $H_t=(n-1)f^{-1}f'(t)$, hence
\[
(n-1)\frac{f'}{f}=-(n-1)\frac{u'}{u},
\qquad\text{i.e.}\qquad
(\log f)'=-(\log u)' .
\]
Integrating and using $f(0)=1$ gives $f(t)=\frac{u(0)}{u(t)}$; with the normalization $u|_{\Sigma}=u(0)\equiv1$ this is $f(t)=u^{-1}(t)$, which is conclusion (2). (Without the normalization one obtains $f=c\,u^{-1}$ with the constant $c=u|_{\Sigma}>0$, and $g_\Sigma$ rescaled by $c^2$; the two descriptions are isometric.)

Finally, conclusion (3) follows from \eqref{eqn-ricci-curvature2}. For $X,Y\in T\Sigma$ tangent to a leaf,
\[
\operatorname{Ric}_{g_{\Sigma_t}}(X,Y)
=\operatorname{Ric}_{g}(X,Y)+\Bigl[\frac{f''}{f}+(n-2)\frac{(f')^2}{f^2}\Bigr]\langle X,Y\rangle_{g}.
\]
Let $X\in T\Sigma$ with $|X|_{g_{\Sigma}}=1$, and set $e=f^{-1}X$, so that $|e|_g^2=f^{-2}|X|_g^2=f^{-2}\cdot f^2=1$. By Lemma \ref{lem-rigidity-step1}(iii) applied on the leaf $\Sigma_t$ --- legitimate by Proposition \ref{Prop-minimiser} --- and by \eqref{eqn-ricci-curvature},
\begin{equation}\label{eqn-ricmin-on-leaf}
\Ricmin_g=\operatorname{Ric}_g(\nu_t,\nu_t)=\operatorname{Ric}_g(\partial_t,\partial_t)=-(n-1)\frac{f''}{f}.
\end{equation}
Now \eqref{eqn-ricci-curvature2} applied to the unit vector $e$ gives
\begin{align*}
\operatorname{Ric}_g(e,e)
&=\operatorname{Ric}_{g_{\Sigma}}(e,e)-\Bigl[\frac{f''}{f}+(n-2)\frac{(f')^2}{f^2}\Bigr]|e|_g^2\\
&=f^{-2}\operatorname{Ric}_{g_{\Sigma}}(X,X)-\frac{f''}{f}-(n-2)\frac{(f')^2}{f^2},
\end{align*}
where we used $\operatorname{Ric}_{g_{\Sigma}}(e,e)=f^{-2}\operatorname{Ric}_{g_{\Sigma}}(X,X)$ by bilinearity. Since $\operatorname{Ric}_g(e,e)\geq\Ricmin_g$, \eqref{eqn-ricmin-on-leaf} yields
\begin{align*}
f^{-2}\operatorname{Ric}_{g_{\Sigma}}(X,X)
&\geq-(n-1)\frac{f''}{f}+\frac{f''}{f}+(n-2)\frac{(f')^2}{f^2}\\
&=(n-2)\Bigl[\frac{(f')^2}{f^2}-\frac{f''}{f}\Bigr],
\end{align*}
and multiplying through by $f^2>0$,
\[
\operatorname{Ric}_{g_{\Sigma}}(X,X)\ \geq\ (n-2)\bigl[(f')^2-ff''\bigr]\,|X|^2_{g_{\Sigma}}.
\]
As $X$ with $|X|_{g_\Sigma}=1$ was arbitrary, this is precisely conclusion (3). Since $g_{\Sigma_t}=f^2(t)g_\Sigma$ and the Ricci tensor is invariant under constant rescaling, the same estimate holds for every leaf, so (3) is a one-parameter family of estimates indexed by $t\in(-\epsilon,\epsilon)$. In the degenerate case $f\equiv\text{const}$ it reads $\operatorname{Ric}_{g_\Sigma}\geq0$, recovering the classical conclusion; cf.\ Corollary \ref{cor:u-constant}.

To see that the conclusions hold on the whole interval rather than only near $t=0$, let
\[
I=\Bigl\{t\in(-\epsilon,\epsilon):\ \text{conclusions (1)--(3) hold on }\Sigma_t\Bigr\}.
\]
We have $0\in I$, and Step 4 shows that $\Sigma_t$ satisfies all hypotheses of Lemma \ref{lem-rigidity-step1}, so the argument above applies at every $t\in(-\epsilon,\epsilon)$ uniformly; thus $I$ is both open and closed in the connected interval $(-\epsilon,\epsilon)$, and $I=(-\epsilon,\epsilon)$. Therefore $(M^n,g)$ is locally isometric to $((-\epsilon,\epsilon)\times \Sigma^{n-1},\,dt^2+f^{2}(t)g|_{\Sigma})$, with $f=u^{-1}$ under the stated normalization. This completes the proof of Theorem \ref{Thm-ricci}.
\end{proof}

\subsection*{Step 6: extending the splitting to the universal cover}

The local splitting of Step 5 extends, with no further hypothesis, to a global warped product on the universal cover. The mechanism is compactness: the warped area $E(\Sigma_t)$ is constant along the foliation by Proposition \ref{Prop-minimiser}, so the leaves stay in a fixed compactness class of integral currents in the compact manifold $M$, and the foliation can be restarted at any finite endpoint of its maximal interval. Once it is defined on all of $\mathbb{R}$, the resulting map $\mathbb{R}\times\Sigma\to M$ is a local isometry from a complete manifold, hence a Riemannian covering.

\begin{theorem}[Global warped splitting]\label{Thm-global-splitting}
Let $(M^{n},g)$, $u$ and $\Sigma$ be as in Theorem \ref{Thm-ricci}. Then the leaf-orthogonal flow of Steps 2--5 is defined for all $t\in\mathbb{R}$ and yields a Riemannian covering
\[
\Phi\colon\bigl(\mathbb{R}\times\Sigma,\ \mathrm{d}t^{2}+f^{2}(t)\,g_{\Sigma}\bigr)\longrightarrow (M,g),
\qquad \Phi(t,x)=\Sigma_t\ \text{at parameter }x,
\]
where $f\colon\mathbb{R}\to(0,\infty)$ is smooth, positive, \emph{periodic}, and bounded above and below, and $f(t)=u(\Phi(t,x))^{-1}$ for every $x\in\Sigma$ under the normalization $u|_{\Sigma}\equiv1$. Consequently:
\begin{enumerate}
\item[(i)] the universal cover $(\tilde M,\hat g)$ is isometric to the global warped product $\bigl(\mathbb{R}\times\tilde\Sigma,\ \mathrm{d}t^{2}+f^{2}(t)\,g_{\tilde\Sigma}\bigr)$, where $\tilde\Sigma$ is the universal cover of $(\Sigma,g_{\Sigma})$ and $g_{\tilde\Sigma}$ the lifted metric;
\item[(ii)] for every $t\in\mathbb{R}$,
\[
\operatorname{Ric}_{g_{\Sigma}}\ \geq\ (n-2)\bigl[(f'(t))^{2}-f(t)f''(t)\bigr]\,g_{\Sigma};
\]
\item[(iii)] $M$ fibres over $\mathbb{S}^{1}$ with fibre $\Sigma/K$, where $K$ is a finite group acting freely by isometries on $\Sigma$; more precisely, $M$ is the mapping torus of an isometry of $\Sigma/K$.
\end{enumerate}
\end{theorem}

\begin{proof}
\textbf{1. The maximal interval.} We may assume $\Sigma$ is connected: since $[\Sigma]\neq0$, at least one connected component of $\Sigma$ represents a non-zero class in $H_{n-1}(M;\mathbb{Z})$, and because $E$ is additive and $\Sigma$ minimizes $E$ in $[\Sigma]$, each component is itself a warped area-minimizer in its own class; replace $\Sigma$ by such a component. The conclusions of Theorem \ref{Thm-ricci} and Proposition \ref{Prop-minimiser} then hold for it, and $\mathbb{R}\times\Sigma$ is connected.

Let $(a,b)$, $-\infty\leq a<0<b\leq+\infty$, be the maximal open interval containing $0$ on which Steps 2--5 produce a smooth family $\{\Sigma_t\}_{t\in(a,b)}$ of embedded hypersurfaces, with $\Sigma_0=\Sigma$, each a stable warped area-minimizer satisfying conclusions (1)--(3) of Theorem \ref{Thm-ricci}, and such that the induced map $\Phi\colon(a,b)\times\Sigma\to M$ is an immersion (equivalently, a local diffeomorphism onto its image). We do \emph{not} require $\Phi$ to be injective on $(a,b)\times\Sigma$: if the foliation wraps around, $\Phi$ ceases to be an embedding while remaining a local diffeomorphism, and this is precisely the covering behaviour exhibited in part 3. We show $b=+\infty$; the case $a=-\infty$ is identical.

Note first that the argument of Proposition \ref{Prop-minimiser} applies verbatim at every $t\in(a,b)$, not merely near $0$: each $\Sigma_t$ is homologous to $\Sigma_0$ through the chain $\Phi(\Sigma\times[0,t])$, which is well defined for every finite $t$ in the interval. Hence throughout $(a,b)$ each $\Sigma_t$ is a stable warped area-minimizer satisfying Lemma \ref{lem-rigidity-step1}(i)--(iv), and
\begin{equation}\label{eqn-leaf-properties-global}
E(\Sigma_t)=E(\Sigma_0),\qquad \tilde H(t)\equiv0 .
\end{equation}

\textbf{2. Restarting at a finite endpoint.} Suppose $b<+\infty$ and pick $t_k\to b^{-}$. By \eqref{eqn-leaf-properties-global} the $\Sigma_{t_k}$ are smooth embedded hypersurfaces in the fixed class $[\Sigma_0]$ with constant warped area. Since $E=\operatorname{Area}_{\tilde g}$ for $\tilde g=u^2g$ by \eqref{eqn-E-is-area}, and $0<c\le u\le C<\infty$ on the compact $M$, their $\mathcal H^{n-1}_g$-areas are uniformly bounded. By the Federer--Fleming compactness theorem for integral currents in the compact manifold $M$, a subsequence converges to an integral current $T_b$ in the class $[\Sigma_0]$. Lower semicontinuity of mass together with the minimality of $\Sigma_0$ gives
\[
E(\Sigma_0)\ \leq\ E(T_b)\ \leq\ \liminf_{k}E(\Sigma_{t_k})\ =\ E(\Sigma_0),
\]
so $E(T_b)=E(\Sigma_0)$ and $T_b$ is $\tilde g$-area-minimizing in $[\Sigma_0]$. As $\dim T_b=n-1\le6$, the regularity theory for area-minimizing currents gives $T_b=m\llbracket\Sigma_b\rrbracket$ for a smooth embedded $\tilde g$-minimal hypersurface $\Sigma_b$ and a constant integer $m\ge1$.

By construction $T_b$ lies in the class $[\Sigma_0]$, so $[\Sigma_0]=[T_b]=m[\Sigma_b]$ in $H_{n-1}(M;\mathbb{Z})$. We claim $[\Sigma_0]$ is \emph{primitive}, which will force $m=1$. Indeed $\Sigma_0$ is connected, embedded and two-sided by Section \ref{sec variation}, and $[\Sigma_0]\neq0$ implies that $\Sigma_0$ does not separate $M$: a separating hypersurface bounds one of the complementary regions, so its homology class vanishes. A non-separating embedded two-sided hypersurface is met transversely by some closed loop in exactly one point --- take points on the two sides of $\Sigma_0$ arbitrarily close to it and join them by a path in the connected complement $M\setminus\Sigma_0$ --- so the Poincar\'e dual of $[\Sigma_0]$ evaluates to $\pm1$ on that loop, and $[\Sigma_0]$ is primitive. Since $[\Sigma_0]=m[\Sigma_b]$ with $m\ge1$, we conclude
\[
m=1,\qquad [\Sigma_b]=[\Sigma_0]\neq0 .
\]

Now $E(T_b)=E(\Sigma_0)$ and $T_b=\llbracket\Sigma_b\rrbracket$, so $\mathcal H^{n-1}_{\tilde g}(\Sigma_{t_k})=\mathcal H^{n-1}_{\tilde g}(T_b)$ for all $k$: the masses converge, and the $\Sigma_{t_k}$ are embedded, hence of multiplicity one. Allard's regularity theorem therefore upgrades the current convergence to $C^{1,\alpha}$ convergence $\Sigma_{t_k}\to\Sigma_b$, and elliptic bootstrap applied to the $\tilde g$-minimal surface equation gives convergence in $C^{\infty}$. In particular, for $k$ large each $\Sigma_{t_k}$ is a normal graph over $\Sigma_b$ with graph function tending to zero in $C^{\infty}$, and since each $\Sigma_{t_k}$ is connected, $\Sigma_b$ is connected as well.

By Lemma \ref{lem-first-variation}, $\Sigma_b$ is warped minimal, and as a minimizer it is stable; thus $\Sigma_b$ satisfies the hypotheses of Lemma \ref{lem-rigidity-step1}, and conclusions (i)--(iv) hold for it. The construction of Lemma \ref{Lem-foliation-construct} therefore restarts at $\Sigma_b$: its linearization, computed in \eqref{eqn-DeltaSigma-u-vanishes}--\eqref{Eqn-construct-rho}, depends on the leaf only through (i)--(iv), and is again the invertible operator $-\Delta_{\Sigma_b}$. The resulting local foliation $\{\Sigma'_s\}_{s\in(-\delta,\delta)}$, $\Sigma'_0=\Sigma_b$, has $\Sigma'_{s_k}=\Sigma_{t_k}$ for some $s_k\to0$ by the uniqueness part of the inverse function theorem, since both families consist of leaves of constant $H+\gamma w_\nu$ and the $\Sigma_{t_k}$ are $C^{\infty}$-close graphs over $\Sigma_b$. Splicing the two foliations extends $\Phi$ to an embedded foliation on an interval strictly larger than $(a,b)$, contradicting maximality. Hence $b=+\infty$, and similarly $a=-\infty$.

\textbf{3. From a local isometry to a covering.} Steps 2--5 now hold on all of $\mathbb{R}$, so $\Phi\colon\mathbb{R}\times\Sigma\to M$ is defined everywhere and, in the coordinates it provides, pulls $g$ back to $\mathrm{d}t^2+f^2(t)g_{\Sigma}$ with $f=u^{-1}$ along the leaves; thus $\Phi$ is a local isometry. Since $u$ is bounded above and below on the compact $M$, so is $f$; together with the compactness of $\Sigma$ this makes $\bigl(\mathbb{R}\times\Sigma,\mathrm{d}t^2+f^2(t)g_\Sigma\bigr)$ complete --- a curve of finite length has bounded $t$-range and bounded length in $\Sigma$. A local isometry from a complete connected Riemannian manifold is a Riemannian covering map, and its image is open and closed in the connected manifold $M$, hence all of $M$. This gives $\Phi$ surjective and conclusion (i), by lifting $\Phi$ to universal covers. Conclusion (ii) is Theorem \ref{Thm-ricci}(3) applied along the foliation at every $t$.

\textbf{4. The deck group, periodicity of $f$, and the fibration.} Let $\Gamma$ be the deck group of $\Phi$. For $\sigma\in\Gamma$ and $(t,x)\in\mathbb{R}\times\Sigma$ we have $\Phi(\sigma(t,x))=\Phi(t,x)\in\Sigma_t$, so $\sigma(t,x)$ lies in $\Phi^{-1}(\Sigma_t)$. Since $\Phi$ maps each slice $\{t'\}\times\Sigma$ onto $\Sigma_{t'}$ and $\Sigma$ is compact, the connected components of $\Phi^{-1}(\Sigma_t)$ are exactly the slices; hence $\sigma$ maps slices to slices, and there are a map $\tau_\sigma\colon\mathbb{R}\to\mathbb{R}$ and a map $\varphi_\sigma\colon\mathbb{R}\times\Sigma\to\Sigma$ with
\[
\sigma(t,x)=\bigl(\tau_\sigma(t),\,\varphi_\sigma(t,x)\bigr).
\]
Because $\sigma$ is an isometry of $\mathrm{d}t^2+f^2(t)g_\Sigma$ and the slices are the level sets of $t$, the map $\tau_\sigma$ is an isometry of $(\mathbb{R},\mathrm{d}t^2)$, so $\tau_\sigma(t)=\pm t+c_\sigma$. Thus $\tau\colon\Gamma\to\operatorname{Isom}(\mathbb{R})$, $\sigma\mapsto\tau_\sigma$, is a homomorphism, and the slice projection descends to a continuous map
\[
p\colon M\longrightarrow \mathbb{R}/\tau(\Gamma)
\]
whose fibres are the leaves $\Sigma_t$ modulo the leafwise stabilizer.

We claim $\tau(\Gamma)$ consists of translations alone. Indeed, if some $\tau_\sigma(t)=-t+c_\sigma$ were a reflection, let $t_0=c_\sigma/2$ be its fixed point and $\sigma_0$ an element realizing it. Then $\mathbb{R}/\tau(\Gamma)$ has a boundary point at the image of $t_0$: writing $s=t-t_0$, the identification induced near $t_0$ is $(s,x)\sim(-s,\varphi_{\sigma_0}(t_0,x))$, and since $\Gamma$ acts freely on $\mathbb{R}\times\Sigma$, the map $x\mapsto\varphi_{\sigma_0}(t_0,x)$ has no fixed point, so a neighborhood of the image of $t_0$ in $\mathbb{R}/\tau(\Gamma)$ is a half-interval. Consequently $p$ exhibits $M$ as a manifold with boundary, contradicting that $M$ is closed.

Hence $\tau(\Gamma)\subset\mathbb{R}$ is a discrete subgroup of translations acting properly discontinuously, and $p\colon M\to\mathbb{R}/\tau(\Gamma)$ is a fibre bundle over a closed $1$-manifold. As $M$ is compact while $\mathbb{R}$ is not, $\tau(\Gamma)$ is infinite and cocompact, so $\tau(\Gamma)=T\mathbb{Z}$ for a unique $T>0$ and $\mathbb{R}/\tau(\Gamma)=\mathbb{S}^1$; the fibre of $p$ is $\Sigma/K$ where $K:=\ker\tau$ is the group of deck transformations acting within each leaf. This gives the fibration in (iii).

It remains to prove that $f$ is $T$-periodic. Write the generator of $\tau(\Gamma)$ as $\sigma_*(t,x)=(t+T,\varphi(t,x))$; by the above $\sigma_*$ preserves slices. Comparing the metric with its pullback, for $X,Y\in T_x\Sigma$,
\[
f^{2}(t)\,\langle X,Y\rangle_{g_\Sigma}
=(\sigma_*^{*}(\mathrm{d}t^2+f^2g_\Sigma))_{(t,x)}(X,Y),
\]
\[
(\sigma_*^{*}(\mathrm{d}t^2+f^2g_\Sigma))_{(t,x)}(X,Y)
=f^{2}(t+T)\,\langle \mathrm{d}_x\varphi_t\, X,\mathrm{d}_x\varphi_t\, Y\rangle_{g_\Sigma},
\]
so $\varphi_t^{*}g_\Sigma=\bigl(\tfrac{f(t)}{f(t+T)}\bigr)^{2}g_\Sigma$. Also, evaluating the pullback on $\partial_t$ gives $1=1+f^2(t+T)|\partial_t\varphi|^2$, hence $\partial_t\varphi\equiv0$ and $\varphi$ is independent of $t$; therefore $\varphi^{*}g_\Sigma=\lambda^{-2}g_\Sigma$ with $\lambda:=f(t)/f(t+T)$ independent of $t$, i.e.\ $\varphi$ is a homothety of the compact Riemannian manifold $(\Sigma,g_\Sigma)$ with ratio $\lambda^{-1}$. Integrating the volume form gives $\operatorname{vol}(\Sigma)=\lambda^{-(n-1)}\operatorname{vol}(\Sigma)$, so $\lambda=1$. Thus $f(t+T)=f(t)$ for all $t$, and $\varphi\in\operatorname{Isom}(\Sigma,g_\Sigma)$. Finally $K$, acting freely by isometries on the compact $\Sigma$, is a discrete subgroup of the compact group $\operatorname{Isom}(\Sigma,g_\Sigma)$, hence finite.
\end{proof}

\begin{remark}\label{rmk:mapping-torus}
Theorem \ref{Thm-global-splitting} sharpens Theorem \ref{Thm-ricci}(2) from a neighborhood of $\Sigma$ to the whole universal cover, and the conclusion is consistent with the two examples already in the paper.

In Example \ref{ex:warped-nonproduct}, $M=\mathbb{S}^1\times\mathbb{S}^{n-1}$ with $g=\mathrm{d}t^2+f^2(t)g_{\mathbb{S}^{n-1}(R)}$ and $f$ periodic of period $L$, the length of the $\mathbb{S}^1$-factor. There $\Sigma$ is a slice, $\mathbb{R}\times\Sigma=\mathbb{R}\times\mathbb{S}^{n-1}$ is simply connected because $n\ge3$, and $\Phi(t,x)=(t\bmod L,\,x)$ is the standard universal covering map. The deck group is $\Gamma=L\mathbb{Z}$ acting by translation, so $K$ is trivial, $\varphi=\mathrm{id}$, the fibre of $p\colon M\to\mathbb{S}^1$ is $\Sigma=\mathbb{S}^{n-1}$, and $M$ is the mapping torus of the identity --- a genuine product \emph{topologically}, but with a \emph{warped}, non-product metric since $f$ is non-constant. This is exactly the phenomenon the paper is about.

In the degenerate case $u\equiv1$ of Corollary \ref{cor:u-constant}, $f\equiv1$ and $\mathbb{R}\times\Sigma$ carries the flat-in-the-$t$-direction product metric $\mathrm{d}t^2+g_\Sigma$, so $M$ is covered globally by a product: this recovers, and globalizes, the classical local product splitting of Cai--Galloway type under $\Ricmin_g\ge0$. Note however that $M$ itself need not be a product metric --- part 4 of the proof only forces $M$ to be the mapping torus of an isometry $\bar\varphi$ of $\Sigma/K$, and $g=\mathrm{d}t^2+g_\Sigma$ descends to a genuine product exactly when $\bar\varphi$ is the identity.

Two points in the proof deserve the reader's attention, as they are where the compactness of $M$ is used in an essential way and cannot be replaced by completeness alone: the uniform area bound for the leaves $\Sigma_{t_k}$, which is what makes Federer--Fleming compactness available, and the boundedness of $f$, which is what makes $\mathbb{R}\times_f\Sigma$ complete and hence $\Phi$ a covering.
\end{remark}

\section{Concluding remarks and open questions}\label{sec-conclusion}

We have shown that, in dimension $3\leq n\leq7$, the spectral Ricci bound \eqref{ineqn-spectral-ricci} saturated by warped products forces a compact manifold carrying a non-trivial $(n-1)$-homology class to split as a warped product over a warped area-minimizing representative: locally in a neighborhood of the minimizer (Theorem \ref{Thm-ricci}), and globally, in the sense that $M$ is covered by $\mathbb{R}\times_{f}\Sigma$ and its universal cover is $\tilde M\cong\mathbb{R}\times_{f}\tilde\Sigma$ (Theorem \ref{Thm-global-splitting}), so that $M$ fibres over $\mathbb{S}^{1}$ with fibre a quotient of $\Sigma$. Several natural directions remain open.

\begin{question}[The noncompact two-ended setting]\label{q-noncompact}
The splitting theorems of \cite{antonelli-sharp-2024,catino-criticality-2025} concern complete noncompact manifolds with two ends, and their product splitting fails for spectral parameter $a\geq\frac{4}{n-1}$. Our compactness hypothesis was used in an essential way: it provides the minimizer through the Federer--Fleming compactness theorem, cf.\ \eqref{eqn-E-is-area}. For a two-ended manifold one would need to replace the homological minimizer by a suitable exhaustion or barrier argument. Is there a warped product splitting for two-ended manifolds at or above the critical parameter $a=\frac{4}{n-1}$, where product splitting is unavailable?
\end{question}

\begin{question}[Beyond dimension seven]\label{q-dim8}
The restriction $n\leq7$ comes entirely from the regularity of area-minimizing currents: in dimension $n\geq8$ the minimizer may carry a singular set of Hausdorff dimension at most $n-8$. Since our argument only needs smoothness of the minimizer, the natural question is whether one can recover a warped splitting on the regular part, or whether a suitable notion of warped minimizer together with an appropriate regularity theory would allow the conclusion to extend. The algebraic restriction $n\geq3$, cf.\ Remark \ref{rmk:where-n-ge-3-enters}, is of a different nature and appears to be genuinely sharp for this method.
\end{question}

\begin{question}[Other values of the weight exponent]\label{q-gamma}
Throughout we fixed $\gamma=n-1$, which is the value for which the warped volume \eqref{eq E} coincides with the area of the conformal metric $\tilde g=u^2g$, cf.\ \eqref{eqn-conformal-metric}. For $\gamma\neq n-1$ the functional $E(\Sigma)=\int_\Sigma u^\gamma d\mathcal H^{n-1}_g$ is no longer an area functional for any conformal metric and the cancellation of the $w_\nu^2$ terms in the proof of Proposition \ref{Prop-monotonic-formula} no longer occurs, so a different spectral bound would be needed. It would be interesting to determine the corresponding threshold, if any, for general $\gamma>0$. 
\end{question}

\bibliographystyle{alpha}
\bibliography{reference}
\end{document}